\documentclass[11pt]{amsart}
\usepackage{color}
\usepackage{amsfonts}
\usepackage{amsxtra}
\usepackage{latexsym}
\usepackage{newlfont}
\usepackage{verbatim}
\usepackage{latexsym,bm,amsmath,amssymb,amsfonts,amsthm}
\newcommand{\q}{\quad}
\newcommand{\qq}{\quad\quad}

\newcommand{\sgn}{\textup{sgn}\, }

\newcommand{\essinf}{\mathop{\textup{essinf}}}

\def\rr{{\mathbb R}}

\def\rz{{{\rr}^n}}

\def\az{\alpha}

\def\bz{\beta}

\def\dz{\delta}

\def\lz{\lambda}

\def\wz{\omega}
\def\l{\left}
\def\r{\right}

\newtheorem{thm}{Theorem}[section]
\newtheorem{lem}{Lemma}[section]

\newtheorem{rmk}{Remark}[section]

\newtheorem{dfn}{Definition}[section]

\begin{document}

\title
[ commutators on weighted Morrey spaces ] {Necessary and sufficient
conditions for boundedness of commutators of the general fractional
integral operators on weighted Morrey spaces}

\author{Zengyan Si}

\address{ Zengyan Si\\
School of Mathematics and Information Science\\
Henan Polytechnic University\\
Jiaozuo 454000 \\
P. R. China}

\email{sizengyan@yahoo.cn}

\author{Fayou Zhao$^*$}

\address{Fayou Zhao (Corresponding author)\\
Department of Mathematics\\
Shanghai University\\
Shanghai 200444\\
P. R. China}

\email{zhaofayou2008@yahoo.com.cn}

\thanks{The second
author is the corresponding author. The research was supported by
Shanghai Leading Academic Discipline Project (Grant No. J50101).}




\begin{abstract}
We prove that $b$ is in $Lip_{\bz}(\bz)$ if and only if the
commutator $[b,L^{-\alpha/2}]$ of the multiplication operator by $b$
 and the general fractional integral operator $L^{-\alpha/2}$ is bounded from the
 weighed Morrey space $L^{p,k}(\omega)$ to $L^{q,kq/p}(\omega^{1-(1-\alpha/n)q},\omega)$,
 where $0<\beta<1$, $0<\alpha+\beta<n, 1<p<{n}/({\alpha+\beta})$,
 ${1}/{q}={1}/{p}-{(\alpha+\beta)}/{n},$
$0\leq k<{p}/{q},$ $\omega^{{q}/{p}}\in A_1$ and $ r_\omega>
\frac{1-k}{p/q-k},$ and here $r_\omega$ denotes the critical index
of $\omega$ for the reverse H\"{o}lder condition.
\end{abstract}

\maketitle
\section{Introduction and main results}

 Suppose that $L$ is a linear operator on $L^2 (\rz)$ which generates an analytic semigroup
$e^{-tL}$ with a kernel $p_t(x, y)$ satisfying a Gaussian upper
bound, that is,
\begin{equation} \label{G} |p_t(x,y)|\leq \frac{C}{t^{{n}/{2}}}e^{-c\frac{|x-y|^2}{t}}\end{equation}
for $x, y\in \rz$ and all $t>0$. Since we assume only upper bound on heat kernel $p_t(x,y)$ and no regularity on its space variables, this property \eqref{G} is satisfied by a class of differential operator,
see \cite{DY1} for details.
\par
For $0<\alpha<n, $ the general fractional integral $L^{-\alpha/2}$
of the operator $L$ is defined by
$$L^{-\frac{\alpha}{2}}f(x)=\frac{1}{\Gamma(\frac{\alpha}{2})}\int_0^\infty e^{-tL}f\frac{dt}{t^{-\alpha/2+1}}(x).$$
Note that if $L = -\Delta$ is the Laplacian on $\rz$, then
$L^{-\alpha/2}$ is the classical fractional integral $I_\alpha$
which plays important roles in many fields.
 Let $b$ be a locally integrable function on $\rz$, the commutator of
$b$ and $L^{-\alpha/2}$ is defined by
$$[b,L^{-\alpha/2}]f(x)=b(x)L^{-\alpha/2}f(x)-L^{-\alpha/2}(bf)(x).$$

For the special case of $L = -\Delta$,  many results have been
produced. Paluszy\'{n}ski \cite{Pa} obtained that $b\in
Lip_{\bz}(\rz)$ if the commutator $[b, I_{\az}]$ is bounded from
$L^p(\rz)$ to $L^r(\rz)$, where $1<p<r<\infty, 0<\bz<1$ and $1/p -
1/r = (\az+\bz)/n$ with $p < n/(\az+\bz)$. Shirai \cite{Sh1} proved
that $b\in Lip_{\bz}(\rz)$ if and only if the commutator  $[b,
I_{\az}]$ is bounded from the classical Morrey spaces
$L^{p,\lz}(\rz)$ to $L^{q,\lz}(\rz)$ for $1<p<q<\infty,\ 0<\az, \
0<\bz<1$ and $0<\az+\bz=(1/p-1/q)(n-\lz)<n$ or $L^{p,\lz}(\rz)$ to
$L^{q,\mu}(\rz)$ for  $1<p<q<\infty,\ 0<\az, \ 0<\bz<1,\
0<\az+\bz=(1/p-1/q)<n,\ 0<\lz<n-(\az+\bz)p$ and $\mu/q=\lz/p$.  Wang
\cite{WH} established some weighted boundedness of properties of
commutator $[b, I_\az]$ on the weighted Morrey spaces $L^{p,k}$
under appropriated conditions on the weight $\wz$, where the symbol
$b$ belongs to (weighted) Lipschitz spaces. The weighted Morrey
space was first introduced by Komori and Shirai \cite{ks}. For the
general case, Wang \cite{WH2} proved that if $b\in Lip_{\bz}(\rz)$,
then the commutator $[b, I_\az]$ is bounded from $L^p(\wz^p)$ to
$L^q(\wz^q)$, where $0<\bz<1,\ 0<\az+\bz<n, \ 1<p<n/(\az+\bz), 1/p -
1/q = (\az+\bz)/n$ and  $\wz^q\in A_1$.

The purpose of this paper is to give necessary and sufficient
conditions for boundedness of commutators of the general fractional
integrals with $b\in Lip_{\bz}(\omega)$ (the weighted Lipschitz
space). Our theorems are the following:

\begin{thm}\label{t3} Let $0<\beta<1$, $0<\alpha+\beta<n, 1<p<\frac{n}{\alpha+\beta}$, ${1}/{q}={1}/{p}-({\alpha+\beta})/{n},$
$0\leq k<\min\{{p}/{q},{p\bz}/{n}\}$ and $\omega^{q}\in
A_1$. Then we have
\par (a) If $b\in Lip_{\bz}(\rz), $ then $[b,L^{-\alpha/2}]$ is bounded from $L^{p,k}(\omega^p, \omega^q)$ to $L^{q,kq/p}(\omega^q)$;
\par (b) If $[b,L^{-\alpha/2}]$ is bounded from $L^{p,k}(\omega^p, \omega^q)$ to $L^{q,kq/p}(\omega^q)$, then $b\in
Lip_{\bz}(\rz)$.
\end{thm}
\begin{thm}\label{t4} Let $0<\beta<1$, $0<\alpha+\beta<n, 1<p<\frac{n}{\alpha+\beta}$, ${1}/{q}={1}/{p}-({\alpha+\beta})/{n},$
$0\leq k<{p}/{q},$ $\omega^{{q}/{p}}\in A_1$ and $ r_\omega>
\frac{1-k}{p/q-k},$  where $r_\omega$ denotes the critical index of
$\omega$ for the reverse H\"{o}lder condition. Then we have
\par (a) If $b\in Lip_{\bz}(\omega), $ then $[b,L^{-\alpha/2}]$ is bounded from $L^{p,k}(\omega)$ to $L^{q,kq/p}(\omega^{1-(1-\alpha/n)q},\omega)$;
\par (b) If $[b,L^{-\alpha/2}]$ is bounded from $L^{p,k}(\omega)$ to $L^{q,kq/p}(\omega^{1-(1-\alpha/n)q},\omega)$, then $b\in
Lip_{\bz}(\omega)$.
\end{thm}

Our results not only extend the results of \cite{WH} from $(-\triangle)$ to a general operator $L$,
but also characterize
  the (weighted) Lipschitz spaces by means of the boundedness of $[b, L^{-\az/2}]$ on the weighted Morrey
  spaces, which extend the results of \cite{WH} and \cite{WH2}.
 The basic tool is based on a modification of sharp maximal function
  $M^\sharp_{L}$ introduced by \cite{Mar}.

  Throughout this paper all notation is standard or will be defined as
needed.  Denote the Lebesgue measure of $B$ by $|B|$ and the weighted
measure of $B$ by $\wz(B)$, where $\wz(B)=\int_B \wz(x)dx$. For a measurable set
$E$, denote by $\chi_E$ the characteristic function of $E$. For a real number $p$, $1<p<\infty$,
let $p'$ be the dual of $p$ such that $1/p+1/{p'}=1$. The letter $C$ will be used for various
constants, and may change from one occurrence to another.
\section{Some preliminaries}

A non-negative function $\wz$ defined on $\rz$ is called weight if it is locally integral. A weight $\wz$ is said to belong to
the Muckenhoupt class $A_p(\rz)$ for $1<p<\infty$, if there exists a constant $C$ such that
\begin{equation*}
   \left(\frac{1}{|B|}\int_B\omega(x)dx\right)\left(\frac{1}{|B|}\int_B\omega(x)^{-\frac{1}{p-1}}dx\right)^{p-1}\leq C.
\end{equation*} for every ball $B\subset\rz$. The class $A_1(\rz)$ is defined replacing the above inequality by

\begin{equation*}
   \left(\frac{1}{|B|}\int_B\omega(x)dx\right)\leq C \essinf_{x\in B}\, \omega(x).
\end{equation*}When $p=\infty, \omega\in A_\infty,$ if there exist positive
constants $\delta$ and $C$ such that given a ball $B$ and $E$ is a
measurable subset of $B$, then
\begin{equation*}
   \frac{\omega(E)}{\omega(B)}\leq C\left(\frac{|E|}{|B|}\right)^\delta .
\end{equation*}

 A weight function $\omega$ belongs to $A_{p,q}$ for $1 < p < q <\infty$
if for every ball $B$ in $\rz$, there exists a positive constant $C$ which is
independent of $B$ such that
\begin{equation*}
   \left(\frac{1}{|B|}\int_B\omega(x)^qdx\right)^{\frac{1}{q}}\left(\frac{1}{|B|}\int_B\omega(x)^{-p'}dx\right)^{\frac{1}{p'}}\leq C.
\end{equation*}

From the definition of  $A_{p,q}$, we can get that \begin{equation}\label{e3}
\wz\in A_{p,q}\ if\ and \ only \ if\ \wz^q \in A_{1+q/{p'}}.
\end{equation}
Since $\wz^{q/p}\in A_1$, then by \eqref{e3}, we have $\wz^{1/p}\in A_{p,q}$.

A weight function $\omega$ belongs to the reverse H\"{o}lder class
$RH_r$ if there exist two constants $r > 1$ and $C > 0$ such that
the following reverse H\"{o}lder inequality
\begin{equation*}
   \left(\frac{1}{|B|}\int_B\omega(x)^rdx\right)^{\frac{1}{r}}\leq C   \frac{1}{|B|}\int_B\omega(x)dx
\end{equation*}
holds for every ball $B$ in $\rz$.

It is well known that if $\omega\in  A_p$ with $1 \leq p < \infty$,
then there exists $r > 1$ such that $\omega\in RH_r.$ It follows
from H\"{o}lder¡¯s inequality that $\omega \in  RH_r$ implies
$\omega \in RH_s$ for all $1 < s < r.$ Moreover, if $\omega\in RH_r,
r > 1,$ then we have $\omega\in RH_{r+\epsilon}$ for some $
\epsilon> 0.$ We thus write $r_w= \sup\{r > 1 : \omega \in RH_r\}$
to denote the critical index of $\omega$ for the reverse H\"{o}lder
condition. For more details on Muchenhoupt class $A_{p,q}$, we refer
the reader to \cite{GC}, \cite{St1}  and  \cite{Tor}.

\begin{dfn}(\cite{ks}) \label{d1} Let $1\leq p<\infty$ and $0\leq k<1$. Then for
two weights $\mu$ and $\nu$, the weighted Morrey space is defined by
\begin{equation*}
    L^{p,k}(\mu,\nu)=\{f\in L_{loc}^p(\mu): \|f\|_{L^{p,k}(\mu,\nu)}<\infty \},
\end{equation*}
where
\begin{equation*}
  \|f\|_{L^{p,k}(\mu, \nu)}=\sup_B\left( \frac{1}{\nu(B)^k}\int_B |f(x)|^p\mu(x)dx\right)^{\frac{1}{p}}.
\end{equation*}
and the supremum is taken over all balls $B$ in $\rz$.
\end{dfn}

If $\mu=\nu,$ then we have the classical Morrey space $L^{p,k}(\mu)$
with measure $\mu$. When $k=0,$ then $L^{p,k}(\mu,\nu)=L^{p}(\mu)$
is the Lebesgue space with measure $\mu$.

\begin{dfn}(\cite{garcia}) \label{d7}Let $1\leq p<\infty$, $0<\beta<1$, and $\omega\in A_{\infty}$.
  A locally integral function $b$ is said to be in $Lip^p_\beta(\omega)$ if
 \begin{equation*}
 \|b\|_{Lip^p_\beta(\omega)}=\sup_B\frac{1}{\omega(B)^{\beta/n}}
 \left( \frac{1}{\omega(B)}\int_B|b(x)-b_B|^p\omega(x)^{1-p}dx\right)^{\frac{1}{p}}\leq C<\infty,
\end{equation*}where $b_B={|B|^{-1}}\int_B b(y)dy$ and the supremum is taken
over all ball $B\subset R^n.$ When $p=1,$ we denote   $Lip^p_\beta(\omega)$ by $Lip_\beta(\omega).$
  \end{dfn}
Obviously, for the case $\wz=1$, then the $Lip^p_{\bz}(\wz)$ space is the classical $Lip^p_\bz$ space.
 \begin{rmk} \label{r1} Let  $\omega \in A_1$, Garc\'{i}a-Cuerva \cite{garcia}
 proved that the spaces $\|f\|_{Lip^p_\beta(\omega)}$ coincide, and the norm of $||\cdot||_{Lip^p_\beta(\omega)}$
  are equivalent with respect to different values of provided that $1\leq p<\infty.$
\end{rmk}

Given a locally integrable function $f$  and $\beta$, $0\leq \bz<n$, define the fractional maximal function by
 $$M_{\beta,r}f(x)=\sup_{x\in B}\left( \frac{1}{|B|^{1-{\beta r}/{n}}}\int_B |f(y)|^rdy\right)^{\frac{1}{r}},\q  r\geq 1,$$
when $0<\bz<n$. If $\bz=0$ and $r=1$, then $M_{0,\ 1}f=Mf$ denotes the usual Hardy-Littlewood maximal function.

   Let $\omega$ be a weight. The weighted maximal operator $M_\omega$ is defined by
  $$M_\omega f(x)=\sup_{x\in B}\frac{1}{\omega(B)}\int_B |f(y)|dy.$$
 The fractional weighted maximal operator $M_{\beta,r,\omega}$ is defined by
  $$M_{\beta,r,\omega}f(x)=\sup_{x\in B}\left( \frac{1}{\omega(B)^{1-{\beta r}/{n}}}
  \int_B |f(y)|^r\omega(y)dy\right)^{\frac{1}{r}},$$
where $0\leq \bz<n$ and $r\geq 1$.  For any $f\in L^p(\rz),\
 p\geq 1,$ the sharp maximal function $M^{\sharp}_Lf$ associated
  the generalized approximations to the identity $\{e^{-tL},\ t>0\}$ is given by Martell \cite{Mar} as follows:
 \begin{equation*}
   M^{\sharp}_Lf(x)=\sup_{x\in B} \frac{1}{|B|}\int_B |f(y)-e^{-t_B L}f(y)|dy,
\end{equation*}
 where $t_B=r^2_B$ and $r_B$ is the radius of the ball $B$.
 For $0<\dz<1$,
we introduce the $\dz-$sharp maximal operator $M_{L,\dz}^\sharp$ as
$$M_{L,\dz}^\sharp f=M_{L}^\sharp(|f|^\dz)^{1/\dz},$$
which is a modification of the sharp maximal operator $M^{\sharp}$
of Fefferman and Stein (\cite{St1}). Set
$M_{\dz}f=M(|f|^\dz)^{1/\dz}$. Using the same methods as those of
\cite{St1} and \cite{P1}, we can get

\begin{lem}\label{l001} Assume that the semigroup $e^{-tL}$ has a kernel $p_t(x,y)$ which satisfies
the upper bound \eqref{G}.  Let $\lz>0$ and $f\in L^p(\rz)$ for some $1<p<\infty$. Suppose that $\omega \in A_\infty$, then for every
  $0<\eta<1$, there exists a real number $\gamma>0$ independent of $\gamma, \ f$ such that  we have the
following weighted version of the local good $\lambda$ inequality,
for $\eta>0$, $A>1$,
 $$\omega\{x\in \rz: M_{\dz} f>A \lambda, M_{L,\dz}^{\sharp}f(x)\leq \gamma\lambda\}\leq \eta \omega\{x\in \rz: M_{\dz}f(x)>\lambda\}.$$
 where $A>1$ is a fixed constant which depends only on $n$.
\end{lem}
If $\mu,\nu\in A_\infty, 1<p<\infty, 0\leq k<1$, then
 \begin{equation}\label{e1}\|f\|_{L^{p,k}(\mu,\nu)}\leq \|M_{\dz}f\|_{L^{p,k}(\mu,\nu)}\leq C\|M^{\sharp}_{_L,
 \dz}f\|_{L^{p,k}(\mu,\nu)}.
 \end{equation}
 In particular, when $\mu=\nu=\omega$ and $\omega \in A_\infty,$ we have
 \begin{equation}\label{e2}\|f\|_{L^{p,k}(\omega)}\leq \|M_{\dz}f\|_{L^{p,k}(\omega)}\leq C\|M^{\sharp}_{_L,
 \dz}f\|_{L^{p,k}(\omega)}.
  \end{equation}

\section{proof of theorem \ref{t3}}
 To prove Theorem \ref{t3}, we need the following lemmas.

\begin{lem}\label{l6} (\cite{DY1}) Assume that the semigroup $e^{-tL}$ has  a kernel $p_t(x,y)$ which satisfies the upper bound \eqref{G}.
Then for $0<\alpha <1,$ the difference operator $L^{-\frac{\alpha}{2}}-e^{-tL}L^{-\frac{\alpha}{2}}$
has an associated kernel $K_{\alpha,t}(x,y)$ which satisfies
$$K_{\alpha,t}(x,y)\leq \frac{C}{|x-y|^{n-\alpha}}\frac{t}{|x-y|^2}.$$
\end{lem}

\begin{lem}\label{l01}(\cite{WH}) Let $0<\az+\bz<n$, $1<p<n/{(\az+\bz)},
\ 1/q=1/p-(\az+\bz)/n$ and $\wz\in A_1$. Then for every $0<k<p/q$
and $1<r<p$, we have
$$\|M_{\az+\bz, r}f\|_{L^{q, kq/p}(\wz^q)}\leq C\|f\|_{L^{p,q}(\wz^p, \wz^q)}.$$
\end{lem}

\begin{lem}\label{l0002}(\cite{ks})  Let $0<\bz<n$, $1<p<n/{\bz},
\ 1/s=1/p-\bz/n$ and $\wz\in A_{p,s}$. Then for every
$0<k<p/s$, we have
$$\|M_{\bz, 1}f\|_{L^{s, ks/p}(\wz^s)}\leq C\|f\|_{L^{p,k}(\wz^p, \wz^s)}.$$
\end{lem}
\begin{lem}\label{l02} (\cite{WH}) Let $0<\az+\bz<n$, $1<p<n/{(\az+\bz)},
\ 1/q=1/p-\az/n$, $1/s=1/q-\bz/n$ and $\wz^q\in A_1$. Then for every
$0<k<p/s$, we have
$$\|M_{\bz, 1}f\|_{L^{s, ks/p}(\wz^s)}\leq C\|f\|_{L^{q,kq/p}(\wz^q, \wz^s)}.$$
\end{lem}
\begin{lem}\label{l03} Let $0<\az+\bz<n$, $1<p<n/{(\az+\bz)},
\ 1/q=1/p-\az/n$, $1/s=1/q-\bz/n$ and $\wz^q\in A_1$. Then for every
$0<k<p\bz/n$, we have
$$\|L^{-\alpha/2}f\|_{L^{q, kq/p}(\wz^q, \wz^s)}\leq C\|f\|_{L^{p,k}(\wz^p, \wz^s)}.$$
\end{lem}
\begin{proof}
As before, we know that $L^{-\alpha/2}f(x)\leq C
I_\alpha(|f|)(x) $ for all $x\in \rz.$ Together with the result (cf.
\cite{WH}), that is,
$$\|I_{\az}f\|_{L^{q, kq/p}(\wz^q, \wz^s)}\leq C\|f\|_{L^{p,k}(\wz^p, \wz^s)},$$
we can get the desired result.
\end{proof}
 \begin{rmk} \label{r2} Using the boundedness property of $I_{\az}$, we also know
 $L^{-\alpha/2}$ is bounded from $L^1$ to weak $L^{n/(n-\az)}$. It is easy to check that Lemma \ref{l01}-\ref{l03} also hold when $k=0$.
\end{rmk}
The following lemma plays an important role in the proof of Theorem
\ref{t3}.

\begin{lem} \label{l04} Let $0<\dz<1,\  0<\alpha <n, \ 0<\beta<1$
and $b\in Lip_\beta(\rz).$ Then for all $r>1$ and for all $x\in
\rz,$ we have
\begin{eqnarray*}
&&M^{\sharp}_{L,\dz}([b, L^{-\alpha/2}]f)(x)\\
&\leq & C \|b\|_{Lip_\beta(\rz)}
\left(M_{\beta,1}(L^{-\alpha/2}f)(x) +M_{\alpha+\beta,
r}f(x)+M_{\alpha+\beta, 1}f(x)\right).
\end{eqnarray*}
\end{lem}
The same method of proof as that of Lemma \ref{l15} (see below), we
omit the details.

\textit{{Proof of Theorem \ref{t3}.}} We first prove $(a)$.  We only prove Theorem \ref{t3}
in the case $0<\az<1$. For the general case $0<\az<n$, the method is the same as that of \cite{DY1}. We omit
the details.

For $0<\az+\bz<n$ and $1<p<n/(\az+\bz)$, we can find a number $r$
such that $1<r<p$. By Eq.\eqref{e2} and Lemma \ref{l04}, we obtain
{\allowdisplaybreaks
\begin{eqnarray*}
&&\|[b,\ L^{-\alpha/2}]f\|_{L^{q, kq/p}(\wz^{q})}\\
&\leq &C \|M^\sharp_{L,\dz}([b,\ L^{-\alpha/2}]f)\|_{L^{q, kq/p}(\wz^{q})}\\
&\leq &C
\|b\|_{Lip_{\bz}(\wz)}\l(\|M_{\beta,1}(L^{-\alpha/2}f)\|_{L^{q,
kq/p}(\wz^{q})}\r.\\
&\q+&\|M_{\alpha+\beta, r}f\|_{L^{q,
kq/p}(\wz^q)}+\l.\|M_{\alpha+\beta, 1}f\|_{L^{q, kq/p}(\wz^{q})}\r).
\end{eqnarray*}}
Let $1/{q_1}=1/p-\az/n$ and $1/q=1/{q_1}-\bz/n$. Since $\wz^q\in
A_1$, then by Eq.\eqref{e3}, we have $\wz\in A_{p, q}$. Since
$0<k<\min\{p/q,\ p\bz/n\}$, by Lemmas \ref{l01}--\ref{l03}, we yield
that
\begin{eqnarray*}
&&\|[b,\ L^{-\alpha/2}]f\|_{L^{q, kq/p}(\wz^{q})}\\
&\leq &C \|b\|_{Lip_{\bz}(\rz)}\l(\|L^{-\alpha/2}f\|_{L^{{q_1},
k{q_1}/p}(\wz^{{q_1}},\ \wz^q)}+\|f\|_{L^{p, k}(\wz^p,\ \wz^q)}\r)\\
&\leq &C \|b\|_{Lip_{\bz}(\rz)}\|f\|_{L^{p, k}(\wz^p,\ \wz^q)}.
\end{eqnarray*}

Now we prove $(b)$.  Let $L = -\Delta$ be the Laplacian on
$\rz$, then $L^{-{\alpha/ 2}}$ is the classical fractional integral
$I_\alpha$. Let $k=0$ and weight $\omega\equiv 1,$  then
$L^{p,k}(\omega^p,\omega^q)=L^p$ and
$L^{q,kq/p}(\omega^q,\omega)=L^q.$ From \cite{Pa}, the $(L^p,
L^q)$ bounedness of $[b,I_\alpha]$ implies that $b\in Lip_{\bz}(\rz)$.

Thus Theorem \ref{t3} is proved.
\qed

\section{proof of theorem \ref{t4}}
We also need some Lemmas to prove Theorem \ref{t4}.

\begin{lem}\label{l1}(\cite{WH}) Let $0<\alpha+\beta <n,
1<p<\frac{n}{\alpha+\beta},{1}/{q}={1}/{p}-{\alpha}/{n}, {1}/{s}={1}/{q}-{\beta}/{n}$
and $\omega^{s/p}\in A_1.$
 Then if $0<k<p/s$ and $r_\omega> \frac{1}{p/q-k}$, we have
$$\|M_{\beta, 1}f\|_{L^{s,ks/p}(\omega^{s/p},\omega)}\leq C\|f\|_{L^{q,kq/p}(\omega^{q/p},\omega)}.$$
\end{lem}

\begin{lem}\label{l2}(\cite{WH}) Let $0<\alpha <n,
1<p<{n}/{\alpha}, {1}/{q}={1}/{p}-{\alpha}/{n}$ and
$\omega^{q/p}\in A_1.$
 Then if $0<k<p/q$ and $r_\omega> \frac{1-k}{p/q-k}$, we have
$$\|M_{\alpha, 1}f\|_{L^{q,kq/p}(\omega^{q/p},\omega)}\leq C\|f\|_{L^{p,k}(\omega)}.$$
\end{lem}

\begin{lem}\label{l3}(\cite{WH}) Let $0<\alpha <n,
1<p<{n}/{\alpha}, {1}/{q}={1}/{p}-{\alpha}/{n}$,
$0<k<p/q$, $\omega\in A_\infty.$ For any $1<r<p,$ we have
$$\|M_{\alpha, r,\omega}f\|_{L^{q,kq/p}(\omega^{q/p},\omega)}\leq C\|f\|_{L^{p,k}(\omega)}.$$
\end{lem}

\begin{lem} \label{l5} Let $0<\alpha <n,
1<p<{n}/{\alpha}, {1}/{q}={1}/{p}-{\alpha}/{n}$ and
$\omega^{q/p}\in A_1.$
 Then if $0<k<p/q$ and $r_\omega> \frac{1-k}{p/q-k}$, we have
$$\|L^{-\alpha/2}f\|_{L^{q,kq/p}(\omega^{q/p},\ \omega)}\leq C\|f\|_{L^{p,k}(\omega)}.$$
\end{lem}
\begin{proof}
Since the semigroup $e^{-tL}$ has a kernel $p_t(x,y)$ which
satisfies the upper bound \eqref{G}, it is easy to check that
$L^{-\alpha/2}f(x)\leq C I_\alpha(|f|)(x) $ for all $x\in
\rz.$ Using the boundedness property of $I_\alpha$ on weighted
Morrey space (cf. \cite{WH}), we have
$$\|L^{-\alpha/2}f\|_{L^{q,kq/p}(\omega^{q/p},\ \omega)}
\leq\|I_\alpha f\|_{L^{q,kq/p}(\omega^{q/p},\ \omega)}  \leq
C\|f\|_{L^{p,k}(\omega)},$$ where $1<p<{n}/{\alpha}$ and
${1}/{q}={1}/{p}-{\alpha}/{n}.$
\end{proof}
\begin{rmk} \label{r3}  It is easy to check that the above lemmas also hold for $k=0$.
\end{rmk}
\begin{lem} \label{l14} Assume that the semigroup $e^{-tL}$ has a kernel $p_t(x,y)$ which satisfies the upper bound \eqref{G},
 and let $b\in Lip_\beta(\omega),\  \omega\in A_1.$ Then, for every function $f\in L^p(\rz),\  p>1, \ x\in \rz, $ and $1<r<\infty, $ we have
 $$\sup_{x\in B} \frac{1}{|B|}\int_B|e^{-t_BL}(b(y)-b_B)f(y)|dy\leq C \|b\|_{Lip_\beta(\omega)}\omega(x)M_{\beta,r,\omega}f(x).$$
\end{lem}
\begin{proof}
Fix $f\in L^p(\rz), 1<p<\infty$ and $x\in B.$ Then
{\allowdisplaybreaks
\begin{eqnarray*}
&&\frac{1}{|B|}\int_B |e^{-t_B L}((b(\cdot)-b_B)f)(y)|dy\\
&\leq& \frac{1}{|B|}\int_B \int_{\rz} |p_{t_B}(y,z)\|(b(z)-b_B)f(z)|dz dy\\
&\leq&\frac{1}{|B|}\int_B \int_{2B} |p_{t_B}(y,z)\|(b(z)-b_B)f(z)|dz dy\\
&\quad +& \frac{1}{|B|}\int_B \sum_{k=1}^\infty\int_{2^{k+1}B\setminus 2^{k}B} |p_{t_B}(y,z)\|(b(z)-b_B)f(z)|dz dy\\
&\doteq &\mathcal{M}+ \mathcal{N}. \end{eqnarray*}}
It follows from  $y\in B$ and $z\in 2B$ that
$$ |p_{t_B}(y,z)|\leq Ct^{-{n}/{2}}_{B}\leq C \frac{1}{|2B|}.$$
Thus, H\"{o}lder's inequality and Definition \ref{d7} lead to
 {\allowdisplaybreaks
\begin{eqnarray*}
\mathcal{M}&\leq& C\frac{1}{|2B|} \int_{2B}|(b(z)-b_B)f(z)|dz\\
&\leq& C \frac{1}{|2B|}\left( \int_{2B}\|b(z)-b_B|^{r'}\omega(z)^{1-r'} dz\right)^{\frac{1}{r'}}
\left( \int_{2B}|f(z)|^r\omega(z)dz\right)^{\frac{1}{r}}\\
&\leq& C \|b\|_{Lip_\beta(\omega)}\frac{1}{|2B|}\omega(2B)^{\frac{\beta}{n}+\frac{1}{r'}} \omega(2B)^{\frac{1}{r}}
\left( \frac{1}{\omega(2B)}\int_{2B}|f(z)|^r\omega(z)dz\right)^{\frac{1}{r}}\\
&\leq& C \|b\|_{Lip_\beta(\omega)}\frac{1}{|2B|}\omega(2B)^{\frac{\beta}{n}+1}
\left( \frac{1}{\omega(2B)}\int_{2B}|f(z)|^r\omega(z)dz\right)^{\frac{1}{r}}\\
&\leq &C \|b\|_{Lip_\beta(\omega)}\omega(x)
\left( \frac{1}{\omega(2B)^{1-\frac{\beta r}{n}}}\int_{2B}|f(z)|^r\omega(z)dz\right)^{\frac{1}{r}}\\
&\leq &C \|b\|_{Lip_\beta(\omega)}\omega(x) M_{\beta,r,\omega}f(x).
\end{eqnarray*}} Moreover, for any $y\in B$ and $z\in 2^{k+1}B\setminus 2^{k}B$, we
have $|y-z|\geq 2^{k-1}r_B$  and $|p_{t_B}|\leq C
\frac{e^{-c2^{2(k-1)}}2^{(k+1)n}}{|2^{k+1}B|}$.
 {\allowdisplaybreaks
\begin{eqnarray*}
\mathcal{N}&=&\frac{1}{|B|}\int_B \sum_{k=1}^\infty\int_{2^{k+1}B\setminus 2^{k}B} |p_{t_B}(y,z)\|(b(z)-b_B)f(z)|dz dy\\
&\leq& C \sum_{k=1}^\infty \frac{e^{-c2^{2(k-1)}}2^{(k+1)n}}{|2^{k+1}B|}\int_{2^{k+1}B} |(b(z)-b_B)f(z)|dz\\
&\leq &C \sum_{k=1}^\infty \frac{e^{-c2^{2(k-1)}}2^{(k+1)n}}{|2^{k+1}B|}\int_{2^{k+1}B} |(b(z)-b_{2^{k+1}B})f(z)|dz\\
&\quad +&C \sum_{k=1}^\infty \frac{e^{-c2^{2(k-1)}}2^{(k+1)n}}{|2^{k+1}B|}\int_{2^{k+1}B} |(b_{2^{k+1}B}-b_{2B})f(z)|dz\\
&\doteq& \mathcal{N}_1+\mathcal{N}_2.
\end{eqnarray*}}
We will estimate the values of terms  $\mathcal{N}_1$ and $\mathcal{N}_2$, respectively.

Using H\"{o}lder's inequality and Remark \ref{r1}, we have
 {\allowdisplaybreaks
\begin{eqnarray*}
\mathcal{N}_1&\leq& C \sum_{k=1}^\infty
\frac{e^{-c2^{2(k-1)}}2^{(k+1)n}}{|2^{k+1}B|}\\
&&\times \left( \int_{2^{k+1}B}|b(z)-b_B|^{r'}\omega(z)^{1-r'}
dz\right)^{\frac{1}{r'}}
\left( \int_{2^{k+1}B}|f(z)|^r\omega(z)dz\right)^{\frac{1}{r}}\\
&\leq& C \sum_{k=1}^\infty
2^{(k+1)n}e^{-c2^{2(k-1)}}\\
&&\times \|b\|_{Lip_\bz(\wz)}\frac{\omega(2^{k+1}B)}{|2^{k+1}B|}
\left( \frac{1}{\omega(2^{k+1}B)^{1-\beta r/n}}\int_{2^{k+1}B}|f(z)|^r\omega(z)dz\right)^{\frac{1}{r}}\\
&\leq& C \|b\|_{Lip_\beta(\omega)}\omega(x) M_{\beta,r,\omega}f(x).
\end{eqnarray*}}

Since $\omega\in A_1,$ by the H\"{o}lder inequality, we get
 {\allowdisplaybreaks
\begin{eqnarray*}
\mathcal{N}_2&\leq& C \sum_{k=1}^\infty 2^{(k+1)n}e^{-c2^{2(k-1)}}\frac{k}{|2^{k+1}B|^{1-\beta r/n}}\omega(x)\|b\|_{Lip_\beta(\omega)}\int_{2^{k+1}B}|f(z)|dz\\
&\leq& C \sum_{k=1}^\infty
k2^{(k+1)n}e^{-c2^{2(k-1)}}\omega(x)\|b\|_{Lip_\beta(\omega)}
\left(\frac{1}{|2^{k+1}B|^{1-\beta r/n}}\int_{2^{k+1}B} |f(z)|^rdz\right)^{\frac{1}{r}}\\
&=& C \sum_{k=1}^\infty
k2^{(k+1)n}e^{-c2^{2(k-1)}}\\
& & \times \omega(x)\|b\|_{Lip_\beta(\omega)}
 \left(\frac{\omega(2^{k+1}B)^{1-\beta r/n}}{|2^{k+1}B|^{1-\beta r/n}}\frac{1}{\omega(2^{k+1}B)^{1-\beta r/n}}\int_{2^{k+1}B} |f(z)|^rdz\right)^{\frac{1}{r}}\\
&\leq &C \sum_{k=1}^\infty k2^{(k+1)n}e^{-c2^{2(k-1)}}
\omega(x)\|b\|_{Lip_\bz(\wz)}
\left(\frac{1}{\omega(2^{k+1}B)^{1-\beta r/n}}\int_{2^{k+1}B} |f(z)|^r \omega(x)dz\right)^{\frac{1}{r}}\\
&\leq& C \|b\|_{Lip_\beta(\omega)}\omega(x) M_{\beta,r,\omega}f(x).
\end{eqnarray*}}

Thus Lemma \ref{l14} is proved.
\end{proof}

\begin{lem} \label{l15} Let $0<\alpha <1,$ $\omega \in A_1$
and $b\in Lip_\beta(\omega).$ Then for all $r>1$ and for all $x\in
\rz,$ we have
\begin{eqnarray*}
&&M^{\sharp}_{L,\dz}([b, L^{-\alpha/2}]f)(x)\leq C \|b\|_{Lip_\beta(\omega)}\\
&\q & \times
\left(\omega(x)^{1+\frac{\beta}{n}}M_{\beta,1}(L^{-\alpha/2}f)(x)
+\omega(x)^{1-\frac{\alpha}{n}}M_{\alpha+\beta,
r,\omega}f(x)+\omega(x)^{1+\frac{\beta}{n}}M_{\alpha+\beta,
1}f(x)\right).
\end{eqnarray*}
\end{lem}
\begin{proof}
For any given $x\in \rz,$ fix a ball $B=B(x_0, r_B)$ which contains
$x.$ We decompose $f=f_1+f_2, $ where $f_1=f\chi_{2B}.$ Observe that
$$[b, L^{-\alpha/2}]f(x)=(b-b_B)L^{-\alpha/2}f-L^{-\alpha/2}(b-b_B)f_1-L^{-\alpha/2}(b-b_B)f_2$$
and
$$e^{-t_B L}([b,L^{-\alpha/2}]f)=e^{-t_B L}[(b-b_B)L^{-\alpha/2}f-L^{-\alpha/2}(b-b_B)f_1-L^{-\alpha/2}(b-b_B)f_2].$$
Then {\allowdisplaybreaks \begin{eqnarray*}
&&\l(\frac{1}{|B|}\int_B \l|[b,L^{-\alpha/2}]f(y)-e^{-t_B L}[b,L^{-\alpha/2}]f(y)\r|^\dz dy\r)^{1/\dz}\\
&\leq & C\l(\frac{1}{|B|}\int_B \l|(b(y)-b_B)L^{-\alpha/2}f(y)dy\r|^\dz\r)^{1/\dz}\\
& \qq &+C\l(\frac{1}{|B|}\int_B \l|L^{-\alpha/2}(b(y)-b_B)f_1)(y)|dy\r|^\dz\r)^{1/\dz}\\
& \qq &+C\l(\frac{1}{|B|}\int_B \l|e^{-t_BL}((b(y)-b_B)L^{-\alpha/2}f)(y)|dy\r|^\dz\r)^{1/\dz}\\
& \qq &+ C\l(\frac{1}{|B|}\int_B \l|e^{-t_BL}L^{-\alpha/2}((b(y)-b_B)f_1(y))|dy\r|^\dz\r)^{1/\dz}\\
& \qq& +C\l(\frac{1}{|B|}\int_B \l|(L^{-\alpha/2}-e^{-t_BL}L^{-\alpha/2})((b(y)-b_B)f_2)(y)|dy\r|^\dz\r)^{1/\dz}\\
&\doteq &I +II+ III +IV+V.
\end{eqnarray*}}

We are going to estimate each term, respectively. Fix $0<\dz<1$ and
choose a real number $\tau$ such that $1<\tau<2$ and $\tau' \dz<1$.
Since $\omega\in A_1,$ then it follows from H\"{o}lder's inequality
that {\allowdisplaybreaks
\begin{eqnarray*}
I &\leq& C\l(\frac{1}{|B|}\int_B \l|(b(y)-b_B)\r|^{\tau
\dz}dy\right)^{\frac{1}{\tau \dz}}
\left(\int_B \l|L^{-\alpha/2}f(y)\r|^{\tau'\dz} dy\right)^{\frac{1}{\tau'\dz}}\\
&\leq& C\l(\frac{1}{|B|}\int_B \l|(b(y)-b_B)\r|dy\right)
\left(\int_B \l|L^{-\alpha/2}f(y)\r| dy\right)\\
&\leq& C\|b\|_{Lip_\beta(\omega)}
\frac{1}{|B|}\omega(B)^{1+{\beta}/{n}}
\left(\int_B \l|L^{-\alpha/2}f(y)\r| dy\right)\\
&\leq &C \|b\|_{Lip_\beta(\omega)}\omega(x)^{1+\bz/n}
M_{\beta,1}(L^{-\alpha/2}f)(x).
\end{eqnarray*}}

 For II,  using H\"{o}lder's inequality and Kolmogorov's inequality(see\cite{GC}, p.485), then we deduce that
 {\allowdisplaybreaks
\begin{eqnarray*}
II &\leq & C \frac{1}{|B|}\int_B |L^{-\alpha/2}(b(y)-b_B)f_1)(y)|dy\\
& \leq & C \frac{1}{|B|} |B|^{\frac{\alpha}{n}}\|L^{-\alpha/2}(b(y)-b_{2B})f_1\|_{L^{\frac{n}{n-\alpha},\infty}}\\
&\leq&C  \frac{1}{|B|^{1-\frac{\alpha}{n}}}\int_B (b(y)-b_{2B})f_1(y)dy\\
&\leq & C
\|b\|_{Lip_\beta(\omega)}\omega(x)^{1-\frac{\alpha}{n}}M_{\alpha+\beta,
r,\omega}f(x).
\end{eqnarray*}}

Using H\"{o}lder's inequality and Lemma \ref{l14}, we obtain that
$$III \leq C  \|b\|_{Lip_\beta(\omega)}\omega(x) M_{\beta,r,\omega}(L^{-\alpha/2}f)(x).$$

For IV, using the estimate in II, we get {\allowdisplaybreaks
\begin{eqnarray*}
IV &\leq &   \frac{C}{|B|}\int_B \int_{2B}|p_{t_B}(y,z)\|b(z)-b_B\|f(z)|dzdy\\
&\leq &  \frac{C}{|2B|} \int_{2B}L ^{-\alpha/2}((b(z)-b_B))f(z)|dz\\
& \leq &  C
\|b\|_{Lip_\beta(\omega)}\omega(x)^{1-\frac{\alpha}{n}}M_{\alpha+\beta,
r,\omega}f(x).
\end{eqnarray*}}

 By virtue of Lemma \ref{l6}, we have
 {\allowdisplaybreaks
\begin{eqnarray*}
V&\leq &  \frac{C}{|B|}\int_B \int_{(2B)^c}|K_{\alpha,t_B}(y,z)\|(b(z)-b_B)f(z)|dzdy\\
&\leq & \frac{C}{|B|}\sum_{k=1}^\infty \int_{2^k r_B\leq |x_0-z|<2^{k+1}r_B}\frac{1}{|x_0-z|^{n-\alpha}}\frac{r^2_B}{|x_0-z|^2}|(b(z)-b_B)f(z)|dz\\
&\leq & C \sum_{k=1}^\infty 2^{-2k} \frac{1}{|2^{k+1}B|^{1-\frac{\alpha}{n}}}\int_{2^{k+1}B} |(b(z)-b_B)f(z)|dz\\
 &\leq& C \sum_{k=1}^\infty 2^{-2k} \frac{1}{|2^{k+1}B|^{1-\frac{\alpha}{n}}}\int_{2^{k+1}B} |(b(z)-b_{2^{k+1}B})f(z)|dz\\
 &\quad & + C \sum_{k=1}^\infty 2^{-2k}(b_{2^{k+1}B}-b_B) \frac{1}{|2^{k+1}B|^{1-\frac{\alpha}{n}}}\int_{2^{k+1}B} |f(z)|dz\\
 &\doteq&  VI +VII.
\end{eqnarray*}}
Making use of the same argument as that of II, we have
$$
VI\leq
C\|b\|_{Lip_\beta(\omega)}\omega(x)^{1-{\alpha}/{n}}M_{\alpha+\beta,
r,\omega}f(x).
$$

Note that $\omega \in A_1,$
$$|b_{2^{k+1}B}-b_{2B}|\leq C k \,\omega (x)\|b\|_{Lip_\beta(\omega)}\omega(2^{k+1}B)^{{\beta}/{n}}.$$
So, the value of $VII$ can be controlled by
$$
C\|b\|_{Lip_\beta(\omega)}
\omega(x)^{1+{\beta}/{n}}M_{\alpha+\beta, 1}f(x).
$$

Combining the above estimates for I--V, we finish the
proof of Lemma \ref{l15}.
\end{proof}

\textit{{Proof of Theorem \ref{t4}.}} We first prove $(a)$.
 As before, we only prove Theorem \ref{t4} in the case
$0<\az<1$. For $0<\az+\bz<n$ and $1<p<n/(\az+\bz)$, we can find a
number $r$ such that $1<r<p$. By Lemma \ref{l15}, we obtain
{\allowdisplaybreaks
\begin{eqnarray*}
&&\|[b,\ L^{-\alpha/2}]f\|_{L^{q, kq/p}(\wz^{1-(1-\az/n)q},\
\wz)}\\
&\leq &C \|M^\sharp_{L,\dz}([b,\ L^{-\alpha/2}]f)\|_{L^{q, kq/p}(\wz^{1-(1-\az/n)q},\
\wz)}\\
&\leq &C
\|b\|_{Lip_{\bz}(\wz)}\l(\|\omega(\cdot)^{1+\frac{\beta}{n}}M_{\beta,1}(L^{-\alpha/2}f)\|_{L^{q,
kq/p}(\wz^{1-(1-\az/n)q},\ \wz)}\r.\\
&\q+&\|\omega(\cdot)^{1-\frac{\alpha}{n}}M_{\alpha+\beta,
r,\omega}f\|_{L^{q, kq/p}(\wz^{1-(1-\az/n)q},\
\wz)}\\
&\q+&\l.\|\omega(\cdot)^{1+\frac{\beta}{n}}M_{\alpha+\beta,
1}f\|_{L^{q, kq/p}(\wz^{1-(1-\az/n)q},\ \wz)}\r)\\
&\leq &C
\|b\|_{Lip_{\bz}(\wz)}\l(\|M_{\beta,1}(L^{-\alpha/2}f)\|_{L^{q,
kq/p}(\wz^{q/p},\ \wz)}\r.\\
&\q+&\|M_{\alpha+\beta, r,\omega}f\|_{L^{q,
kq/p}(\wz)}+\l.\|M_{\alpha+\beta, 1}f\|_{L^{q, kq/p}(\wz^{q/p},\
\wz)}\r).
\end{eqnarray*}}
Let $1/{q_1}=1/p-\az/n$ and $1/q=1/{q_1}-\bz/n$.  Lemmas
\ref{l1}--\ref{l5} yield that
\begin{eqnarray*}
&&\|[b,\ L^{-\alpha/2}]f\|_{L^{q, kq/p}(\wz^{1-(1-\az/n)q},\
\wz)}\\
&\leq &C \|b\|_{Lip_{\bz}(\wz)}\l(\|L^{-\alpha/2}f\|_{L^{{q_1},
k{q_1}/p}(\wz^{{q_1}/p},\ \wz)}+\|f\|_{L^{p, k}(\wz)}\r)\\
&\leq &C \|b\|_{Lip_{\bz}(\wz)}\|f\|_{L^{p, k}(\wz)}.
\end{eqnarray*}

Now we prove (b). Let $L = -\Delta$ be the Laplacian on $\rz$,
then $L^{-{\alpha/ 2}}$ is the classical fractional integral
$I_\alpha$. We use the same argument as Janson \cite{Jan}. Choose $Z_0 \in \rz$ so that $|Z_0|=3.$ For $x\in
B(Z_0,2),$ $|x|^{-\alpha+n}$ can be written as the absolutely
convergent Fourier series, $|x|^{-\alpha+n}=\sum_{m\in Z_n}a_m
e^{i<\nu_m,x>}$ with $\sum_m |a_m|<\infty$ since $|x|^{-\alpha+n}
\in C^\infty(B(Z_0,2))$. For any $x_0\in \rz$ and $\rho>0,$ let
$B=B(x_0,\rho)$ and $B_{Z_0}=B(x_0+Z_0 \rho,\rho),$
{\allowdisplaybreaks \begin{eqnarray*}
&&\int_B |b(x)-b_{B_{Z_0}}|dx=\frac{1}{|B_{Z_0}|}\int_B \l|\int_{B_{Z_0}}(b(x)-b(y))dy\r|dx\\
&=&\frac{1}{\rho^n}\int_B
s(x)\l(\int_{B_{Z_0}}(b(x)-b(y))|x-y|^{-\alpha+n}|x-y|^{n-\alpha}dy
\r)dx,
\end{eqnarray*}
}where $s(x)=\overline{\sgn(\int_{B_{Z_0}}(b(x)-b(y))dy)} .$ Fix
$x\in B$ and $y\in B_{Z_0}$, then ${(y-x)}/{\rho}\in B_{Z_0,2}$,
hence, {\allowdisplaybreaks
\begin{eqnarray*}
&&\frac{\rho^{-\alpha+n}}{\rho^n}\int_B s(x)
\l(\int_{B_{Z_0}}\l(b(x)-b(y)\r)|x-y|^{-\alpha+n}\l(\frac{|x-y|}{\rho}\r)^{n-\alpha}dy \r)dx\\
&= &\rho^{-\alpha}\sum_{m\in Z^n}a_m \int_B s(x)
\l(\int_{B_{Z_0}}\l(b(x)-b(y)\r)|x-y|^{n-\alpha}e^{i<\nu_m,y/\rho>}dy \r)e^{-i<\nu_m,x/\rho>}dx\\
&\leq &\rho^{-\alpha} \l|\sum_{m\in Z^n}|a_m|
\int_B s(x)[b,L^{-{\alpha/ 2}}]\l(\chi_{B_{Z_0}}e^{i<\nu_m,\cdot/\rho>}\r)\chi_B(x)e^{-i<\nu_m,x/\rho>}dx\r| \\
&\leq &\rho^{-\alpha} \sum_{m\in Z^n}|a_m|
\|[b,L^{-{\frac{\alpha}{2}}}](\chi_{B_{Z_0}}e^{i<\nu_m,\cdot/\rho>})\|_{L^{q,0}
(\omega^{1-(1-\alpha/ n)q},\omega)} \l(\int_B \omega(x)^{q'(\frac{1}{{q'}}-\frac{\alpha}{ n})}dx\r)^{\frac{1}{q'}}\\
&\leq &C\rho^{-\alpha} \sum_{m\in Z^n}|a_m|
\|\chi_{B_{Z_0}}\|_{L^{p,0}(\omega)}
\l(\int_B \omega(x)^{{q'}(1/{q'}-\alpha/ n)}dx\r)^{\frac{1}{q'}}\\
&\leq &C \omega(B)^{1/p+1/{q'}-\alpha/ n}= C \omega(B)^{1+\beta/n}.
\end{eqnarray*}
This implies that $b\in Lip_{\bz}(\omega)$. Thus, $(b)$ is proved.
\qed
}

\end{document}